\UseRawInputEncoding
\documentclass[11pt]{article}
\usepackage{float}
\usepackage{color}
\usepackage{subfigure}
\usepackage{amsmath}
\usepackage{amsthm}
\usepackage{amsfonts}
\usepackage{amssymb,latexsym}
\usepackage[all]{xy}
\usepackage{graphicx}
\usepackage[mathscr]{eucal}
\usepackage{verbatim}
\usepackage{hyperref}

\theoremstyle{plain}

    \newtheorem{thm}{Theorem}[section]
       
       \newtheorem{lem}{Lemma}[section]
       \newtheorem{cor}{Corollary}[section]
       \newtheorem{defn}{Definition}[section]
       \newtheorem{rem}{Remark}[section]

\numberwithin{equation}{section}

\usepackage{graphicx}

\begin{document}
\title{Hausdorff and fractal dimensions of  attractors for functional  differential equations in Banach spaces}

\author{Wenjie Hu$^{1,2}$\footnote{Corresponding author.  E-mail address: wenjhu@hunnu.edu.cn (Wenjie Hu)},  Tom\'{a}s
Caraballo$^{3,4}$
\\
\small  1. The MOE-LCSM, School of Mathematics and Statistics,  Hunan Normal University,\\
\small Changsha, Hunan 410081, China\\
\small  2. Journal House, Hunan Normal University, Changsha, Hunan 410081, China\\
\small 3 Dpto. Ecuaciones Diferenciales y An\'{a}lisis Num\'{e}rico, Facultad de Matem\'{a}ticas,\\
\small  Universidad de Sevilla, c/ Tarfia s/n, 41012-Sevilla, Spain\\
\small 4 Department of Mathematics, Wenzhou University, \\
\small  Wenzhou, Zhejiang Province, 325035, China.
}

\date {}
\maketitle

\begin{abstract}
The main objective of this paper is  to obtain estimations of  Hausdorff dimension as well as fractal dimension  of  global attractors and pullback attractors for both autonomous and nonautonomous  functional differential equations (FDEs) in Banach spaces. New criterions  for the finite Hausdorff dimension and fractal dimension of attractors in Banach spaces are firslty proposed by combining the squeezing property and the covering of finite subspace of Banach spaces,  which generalize the method established in Hilbert spaces. In order to surmount the barrier caused by the lack of orthogonal projectors with finite rank, which is the key tool for  proving the squeezing property of partial differential equations in Hilbert spaces, we adopt the state decomposition of phase space based on the exponential dichotomy of the studied FDEs to obtain similar squeezing property. The theoretical results are applied to a retarded nonlinear reaction-diffusion equation and a non-autonomous retarded functional differential equation in the natural phase space, for which explicit bounds of dimensions that  do  not depend on the entropy number but only depend on the spectrum of the linear parts  and  Lipschitz constants of the nonlinear parts are obtained.
\end{abstract}

\bigskip


\section{Introduction}
For infinite dimensional dynamical systems generated by partial differential equations or delay differential equations, the phase spaces are generally not locally compact, while the existence of attractors can  reduce the essential parts of the flows to  compact sets. Furthermore, if the attractors have finite topological dimensions (including Hausdorff dimension and fractal dimension), then the attractors  can be described by a finite number of parameters and hence the dynamics of the infinite dynamical systems are likely to be studied by the concepts and methods of finite dimensional dynamical systems. Owing to this, the study of attractors as well as their topological dimensions estimation have received much attention from pure and applied mathematics community during the past decades.

The theory of  existence of attractors for deterministic infinite dimensional dynamical systems, especially for a large class of parabolic partial differential equations and delay differential equations, is now well developed. See, for instance, the monographs by Babin and Vishik \cite{2}, Hale \cite{HJ}, Ladyzhenskaya \cite{27}, Robinson \cite{34} and Temam \cite{37}. With respect to the dimensions estimation of attractors for infinite dimensional dynamical systems, there are several methods. The first one is the squeeze method in Hilbert spaces, which dates  back  to the pioneering work of  Foias and Temam \cite{T14}, where they showed that under some circumstances a three dimensional flow depends indeed on a finite number of parameters. Then, the idea was adopted in \cite{BV85} and \cite{L85} to show the finite Hausdorff dimensionality of the global attractor for the Navier-Stokes equations, which was then extended to the random case by Debussche in \cite{DA97}. The idea has also been extended to investigate the exponential attractors of deterministic partial differential equations in \cite{BN95,EFK98,EFNT94,M99} and stochastic partial differential equations in \cite{14,ZM}, which implies the finite fractal  dimensions  of attractors.

The second method is to compute traces of some linear operators generated by the linearization  of the equations, requiring quasi-differentials of the underlying systems, which originates from the early work \cite{T8}  with an emphasise on the Hausdorff  dimension. It was then  further extended by Foias and Temam  in their later work \cite{T6}, where they also investigated fractal dimension and the relationship between the Hausdorff  dimension  of attractors and Lyapunov exponents and Lyapunov numbers. The method  has also been adopted to study the  dimensions of attractors for a variety of evolution equations. See, for instance, the nonautonomous PDEs in bounded domain \cite{CV94} and unbounded domain \cite{CE00}, the retarded semilinear partial differential equations in \cite{SW91} and retarded  Navier-Stokes-Voight equation in \cite{QY}. More examples can be found in Babin and Vishik \cite{2} and Temam \cite{37} and the references therein.

The above mentioned dimensions estimation methods are mainly obtained in  Hilbert spaces, i.e., the phase space endowed with a smooth inner product geometrical structure. Nevertheless, there are many evolution equations arising from real world modelings defined in Banach spaces, such as the  delay differential equations \cite{JH}, delay partial differential  equations \cite{WJ} and the non-autonomous Chafee-Infante equation \cite{5} and so on. Although, in \cite{JM76,QY,SW91}, the authors studied dimensions of global attractors for ordinary or partial functional differential equations,  they recast the equations into Hilbert spaces. Therefore, one natural question arises, what can we say about the dimensions of attractors for FDEs in the natural phase space, i.e., the Banach spaces? M\'{a}lek, Ruzicka and  Th\"{a}ter \cite{C25} established a third method for estimating fractal dimension based on a smoothing property of the system, which allows the phase space to be  Banach spaces, but requires an auxiliary space that is compactly embedded into the phase space. The method was then extended by  Efendiev, Miranville  and  Zelik \cite{C17,C18}  to construct  pullback and uniform exponential attractors for systems in Banach spaces, which has also been widely used in estimating the fractal dimension and construct exponential attractors of deterministic systems \cite{C9,CC,C21,C20} and random systems \cite{C5,C32,C37}.

Although, the third method is effective for systems in Banach spaces, the estimation of the fractal dimension depends on the choice of another embedding space which may vary from space to space. Furthermore, the dimension estimation depends on the entropy number between two spaces for which is generally quite difficult to obtain an explicit bound. Hence, one naturally wonders whether we can give explicit bounds of topological dimensions of attractors for systems in Banach spaces that only depend on the inner characteristic of the system. The only works tackling topological dimensions estimation of invariant sets for nonlinear maps or attractors for infinite dimensional dynamical systems in Banach spaces that we can find are \cite{30} and \cite{C6}. In \cite{30}, Ma\~n\'{e} showed that negative invariant sets for certain nonlinear maps in Banach spaces have finite fractal dimension, which also imply the finite dimensionality of Hausdorff dimension and was then improved by \cite{C6}. Nevertheless, in \cite{30}, the results are obtained under the assumption that the derivatives of map are bounded. In this paper, we establish a new method by combining the squeezing property obtained by exponential dichotomy and the covering lemma of the finite dimensional subspace of Banach space established in \cite{30}. We do not need the strict restriction of the boundedness of derivatives and we also provide explicit bounds of the dimensions  of the invariant sets in Banach spaces that only depend on the spectrum of the linearized system and state decomposition of the phase space while not relate  to the entropy number as \cite{C9,CC,C17,C18,C21,C20,C25} did.

We organize the remaining part of this paper  as follows. In Section 2, we establish the criteria for both Hausdorff and fractal dimensions estimation for autonomous systems. The corresponding results for nonautonomous systems are obtained in Section 3 followed by applications to a retarded nonlinear reaction-diffusion equations and a non-autonomous retarded functional differential equations which depend only on the spectrum of the linear part and Lipschitz constants of the nonlinear terms are given  in Section 4. At last, we summarize the paper and point out some potential directions for future research in Section 5.

\section{Dimensions of attractors for autonomous systems}
In this section, we study the Hausdorff and fractal dimensions of the attractors for autonomous dynamical systems. Let $X$ be a Banach space with norm $\|\cdot\|_{X}$   and $S(t): X\rightarrow X, t\geq 0$ be a semigroup. A set $\mathcal{A}$  is said to be invariant, if $S(t)\mathcal{A}=\mathcal{A}$. Moreover, the invariant set $\mathcal{A}$ is said to be a global attractor if $\mathcal{A}$ is a maximal compact invariant set which attracts each bounded set $B \subset X$. The Hausdorff dimension of the compact set $\mathcal{A}\subset X$ is
$$
d_{H}(\mathcal{A})=\inf \left\{d: \mu_{H}(\mathcal{A}, d)= 0 \right\}
$$
where, for $d \geq 0$,
$$\mu_{H}(\mathcal{A}, d)=\lim _{\varepsilon \rightarrow 0} \mu_{H}(\mathcal{A}, d, \varepsilon)$$
 denotes the $d$-dimensional Hausdorff measure of the set $\mathcal{A}\subset X$, where
 $$\mu_{H}(\mathcal{A}, d, \varepsilon)=\inf \sum_{i} r_{i}^{d}$$
 and the infimum is taken over all coverings of $\mathcal{A}$ by balls of radius $r_{i} \leqslant \varepsilon$. It can be shown that there exists $d_{H}(\mathcal{A}) \in[0,+\infty]$ such that $\mu_{H}(\mathcal{A}, d)=0$ for $d>d_{H}(\mathcal{A})$ and $\mu_{H}(\mathcal{A}, d)=\infty$ for $d<d_{H}(\mathcal{A})$. $d_{H}(\mathcal{A})$ is called the Hausdorff dimension of $\mathcal{A}$.

For a finite dimensional subspace $F$ of a Banach space $X$, denote by $B^F_r(x)$ the ball in $F$ of center $x$ and radius $r$, that is $B^F_r(x)=\{y \in F |\|y-x\|\leq r\}$. It is proved in \cite{30} that the following covering lemma of balls in finite dimensional Banach spaces is true.
\begin{lem}\label{lem3.1}
For every finite dimensional subspace $F$ of a Banach space $X$, we have
 \begin{equation}\label{3.1}
N\left(r_1, B_{r_2}^F\right)\leq m 2^m\left(1+\frac{r_1}{r_2}\right)^m,
\end{equation}
for all  $r_1>r_2>0$, where $m=\operatorname{dim} F$ and $N\left(r_1, B_{r_2}^F(0)\right)$ is the minimum number of balls needed to cover the ball of radius $r_1$ by balls $B_{r_2}^F(0)$ of radius $r_2$  calculated in the metric space $X$.
\end{lem}

To prove the existence of finite Hausdorff dimension of the attractors for a semigroup $\{S(t)\}_{t\geq 0}$ in a Banach space, we impose the following additional  assumptions on its attractor $\mathcal{A}$.

$\mathbf{Hypothesis\  A1}$ There is a finite dimensional projection $P:X\rightarrow PX$ with a finite dimension \begin{equation}\label{2.2}
\Lambda=\dim\{PX\},
\end{equation}
and there exist three positive numbers $t_0, M_1, M_2, M_3$  and two constants $\lambda_0$ and $\lambda_1$ such that
 \begin{equation}\label{3.2}
\left\|P S(t_0)\varphi-P S(t_0)\psi\right\| \leq M_1e^{\lambda_0 t_0}\left\|\varphi-\psi\right\|
\end{equation} 
and
 \begin{equation}\label{2.4}
\begin{gathered}
\left\|(I-P) S(t_0)\varphi-(I-P) S(t_0)\psi\right\|\leq (M_2e^{\lambda_1 t_0}+M_3e^{\lambda_0 t_0})\left\|\varphi-\psi\right\|,
\end{gathered}
\end{equation}
for any  $\varphi, \psi$ in $\mathcal{A}$. 

We can now prove an upper bound for the Hausdorff dimension of the attractors for a semigroup $S(t)$ in a Banach space under $\mathbf{Hypothesis\  A1}$.
\begin{thm}\label{thm3.1} Assume that $\{S(t)\}_{t\geq 0}$ is a continuous semigroup with  global attractor $\mathcal{A}$, $\mathbf{Hypothesis\  A1}$ holds and there exists $0<\alpha<2$ such that
 \begin{equation}\label{3.4a}
\alpha M_1e^{\lambda_0 t_0}+2M_2e^{\lambda_1 t_0}+2M_3e^{\lambda_0 t_0}<1.
\end{equation}
Then, the Hausdorff dimension of the global attractor $\mathcal{A}$ satisfies
 \begin{equation}\label{3.4}
d_{H}<\frac{-\ln\Lambda-\Lambda\ln(2+\frac{4}{\alpha})}{\ln (\alpha M_1e^{\lambda_0 t_0}+2M_2e^{\lambda_1 t_0}+2M_3e^{\lambda_0 t_0})},
\end{equation} 
where $\Lambda$ is the dimension of $PX$ defined by \eqref{2.2} and $M_1, M_2, M_3, \lambda_0$ and $\lambda_1$ are given in $\mathbf{Hypothesis\  A1}$.
\end{thm}
\begin{proof}
Since $\mathcal{A}$ is a compact subset of $X$, for any $0<\varepsilon<1$, there exist $r_1, \ldots, r_N$ in $(0, \varepsilon]$ and $\tilde{u}_1, \ldots, \tilde{u}_N$ in $X$ such that
 \begin{equation}\label{3.5}
\begin{gathered}
\mathcal{A}\subset \bigcup_{i=1}^N B\left(\tilde{u}_i, r_i\right),
\end{gathered}
\end{equation}
where $B(\tilde{u}_i, r_i)$ represents the ball in $X$ of center $\tilde{u}_i$ and radius $r_i$. Without loss of generality, we can assume that for any $i$
 \begin{equation}\label{3.6}
\begin{gathered}
B\left(\tilde{u}_i, r_i\right) \cap \mathcal{A} \neq \emptyset,
\end{gathered}
\end{equation}
otherwise, it can be deleted from the sequence $\tilde{u}_1, \ldots, \tilde{u}_N$. Therefore, we can choose  $u_i, i=1,2, \cdots, N$ such that
 \begin{equation}\label{3.7}
\begin{gathered}
u_i \in B\left(\tilde{u}_i, r_i\right) \cap \mathcal{A},
\end{gathered}
\end{equation}
and
 \begin{equation}\label{3.8}
\begin{gathered}
\mathcal{A} \subset \bigcup_{i=1}^N\left(B\left(u_i, 2 r_i\right) \cap \mathcal{A}\right).
\end{gathered}
\end{equation}
It follows from \eqref{3.2} and \eqref{2.4} that for any  $u\in B\left(u_i, 2 r_i\right) \cap \mathcal{A}$, we have
 \begin{equation}\label{3.9}
\begin{gathered}
\left\|P S\left(t_0\right) u-P S\left(t_0\right) u_i\right\| \leq 2M_1e^{\lambda_0 t_0} r_i,
\end{gathered}
\end{equation}
and
 \begin{equation}\label{3.10}
\begin{gathered}
\left\|(I-P) S\left(t_0\right) u-(I-P) S\left(t_0\right) u_i\right\| \leq 2(M_2e^{\lambda_1 t_0}+M_3e^{\lambda_0 t_0})r_i.
\end{gathered}
\end{equation}

By Lemma \ref{lem3.1}, for any $\alpha>0$, we can find $y_i^1, \ldots, y_i^{n_i}$ such that 
 \begin{equation}\label{3.11}
\begin{gathered}
B_{P X}\left(P S\left(t_0\right) u_i, 2 M_1e^{\lambda_0 t_0} r_i\right) \subset \bigcup_{j=1}^{n_i} B_{PX}\left(y_i^j, \alpha M_1e^{\lambda_0 t_0} r_i\right)
\end{gathered}
\end{equation}
with
 \begin{equation}\label{3.12}
\begin{gathered}
n_i \leq \Lambda 2^\Lambda \left(1+\frac{2}{\alpha}\right)^\Lambda,
\end{gathered}
\end{equation} 
where $\Lambda$ is the dimension of $P X$ and we have denoted by $B_{PX}(y, r)$ the ball in $PX$ of radius $r$ and center $y$.

Set
 \begin{equation}\label{3.13}
\begin{gathered}
u_i^j=y_i^j+(I-P) S\left(t_0\right) u_i
\end{gathered}
\end{equation}
for $i=1, \ldots, N, j=1, \ldots, n_i$. Then, for any $u\in B\left(u_i, 2 r_i\right) \cap \mathcal{A}$,  there exists a $j$ such that
 \begin{equation}\label{3.14}
\begin{aligned}
\left\|S\left(t_0\right) u-u_i^j\right\|
& \leq\left\|P S\left(t_0\right) u-y_i^j\right\|+\left\|(I-P) S\left(t_0\right) u-(I-P) S\left(t_0\right) u_i\right\| \\
& \leq\left(\alpha M_1e^{\lambda_0 t_0}+2M_2e^{\lambda_1 t_0}+2M_3e^{\lambda_0 t_0}\right) r_i
\end{aligned}
\end{equation}
with
 \begin{equation}\label{3.16}
n_i \leq \Lambda(2+\frac{4}{\alpha} )^{\Lambda}
\end{equation}
Denote by $\eta=\alpha M_1e^{\lambda_0 t_0}+2M_2e^{\lambda_1 t_0}+2M_3e^{\lambda_0 t_0}$, then we have
 \begin{equation}\label{3.20}
S\left(t_0\right)\left(B\left(u_i, 2 r_i\right) \cap \mathcal{A}\right) \subset \bigcup_{j=1}^{n_i} B\left(u_i^j, \eta r_i\right).
\end{equation} 
Thanks to the invariance of $\mathcal{A}$, i.e., $\mathcal{A}=S\left(t_0\right) \mathcal{A}$, we have
 \begin{equation}\label{3.21}
\mathcal{A} \subset \bigcup_{i=1}^N \bigcup_{j=1}^{n_i} B\left(u_i^j, \eta r_i\right) .
\end{equation}
This implies that, for any $d \geq 0$,
 \begin{equation}\label{3.22}
\begin{aligned}
\mu_H\left(\mathcal{A}, d, \eta\varepsilon\right)
\leq \sum_{i=1}^N \sum_{j=1}^{n_i} \eta^{d}r_i^d \leq \Lambda(2+\frac{4}{\alpha})^{\Lambda} \eta^{d} \sum_{i=1}^N r_i^d,
\end{aligned}
\end{equation}
we deduce, by taking the infimum over all the coverings of $\mathcal{A}$ by balls of radii less than $\varepsilon$,
 \begin{equation}\label{3.23}
\begin{aligned}
\mu_H\left(\mathcal{A}, d, \eta\varepsilon\right)\leq \Lambda(2+\frac{4}{\alpha})^{\Lambda} \eta^{d}\mu_H(\mathcal{A}, d, \varepsilon).
\end{aligned}
\end{equation}
Applying the formula recursively for $k$ times yields
 \begin{equation}\label{2.27}
\begin{aligned}
\mu_H\left(\mathcal{A}, d, (\eta\varepsilon)^k\right)\leq [\Lambda(2+\frac{4}{\alpha})^{\Lambda} \eta^{d}]^k\mu_H(\mathcal{A}, d, \varepsilon).
\end{aligned}
\end{equation}
Therefore, if
 \begin{equation}\label{3.25}
d<\frac{-\ln\Lambda-\Lambda\ln(2+\frac{4}{\alpha})}{\ln (\alpha M_1e^{\lambda_0 t_0}+2M_2e^{\lambda_1 t_0}+2M_3e^{\lambda_0 t_0})},
\end{equation}
then
 \begin{equation}\label{3.31}
\Lambda(2+\frac{4}{\alpha})^{\Lambda} \eta^{d}<1.
\end{equation}
Thus, by taking $k \rightarrow \infty$, we have $(\eta\varepsilon)^k\rightarrow 0$
and \eqref{2.27} leads to
 \begin{equation}\label{3.32}
\mu_H(\mathcal{A}, d, (\eta\varepsilon)^k) \rightarrow 0.
\end{equation}
This completes the proof.
\end{proof}
\begin{rem}\label{rem2.1} 
By \eqref{3.4}, we can see  estimation of $d_H$ depends on the parameter $\alpha$. It would be convenient to know what is the optimal bound for this dimension. In other words, what is the value of $\alpha$ such that the right  hand side of \eqref{3.4} attains its minimum value. By doing some simulations with some specific values of the parameters involved in \eqref{3.4}, one can check that when $M_1e^{\lambda_0t_0}$ becomes smaller and smaller the value of $\alpha$ at which the minimum is achieved is closer and closer to $2$, and from one threshold on this minimum is achieved when $\alpha=2$. In conclusion, in these particular situations, if we  take $\alpha\uparrow 2$ and assume   $2M_1e^{\lambda_0t_0}+2M_2e^{\lambda_1t_0}+2M_3e^{\lambda_0t_0}<1$, then for all $\alpha\in (0,2)$, we have
$\alpha M_1e^{\lambda_0t_0}+2M_2e^{\lambda_1t_0}+2M_3e^{\lambda_0t_0}<1$  and hence we obtain the estimation
 \begin{equation}\label{3.4w}
d_{H}\leq \frac{-\ln\Lambda-\Lambda\ln 4}{\ln (2e^{\lambda_0 t_0}+2M_2e^{\lambda_1 t_0}+2M_3e^{\lambda_0 t_0})},
\end{equation}
which is independent of $\alpha$. But this works only in those mentioned cases in which $M_1e^{\lambda_0t_0}$ is smaller than some threshold (which depends on all the parameters  involved in the problem. 
\end{rem}

Next, we study the fractal dimension of attractors for the semigroup $\{S(t)\}_{t\geq 0}$  based on the state space decomposition and squeeze property $\mathbf{Hypothesis\  A1}$ following  the idea of \cite{DA97,T14} and \cite{ZM}. Here, we extend the method to Banach spaces. The fractal dimension (or capacity) of $\mathcal{A}$ is defined as
 \begin{equation}\label{3.33}
\operatorname{dim}_f \mathcal{A}=\limsup_{\varepsilon \rightarrow 0} \frac{\ln N_{\varepsilon}(\mathcal{A})}{-\ln \varepsilon},
\end{equation}
where $N_{\varepsilon}(\mathcal{A})$ is the minimum number of balls of radius less than $\varepsilon$ which is
necessary to cover $\mathcal{A}$.

\begin{thm}\label{thm3.2}
Let $\mathcal{A}$ be the global attractor of $\{S(t)\}_{t\geq 0}$ with finite diameter $R_\mathcal{A}$, that is $R_\mathcal{A}:=\sup _{u \in \mathcal{A}}\|u\|_X<\infty$.  Assume that $\mathbf{Hypothesis\  A1}$ holds  and  there exists  $0<\alpha<M_1$  such that $\zeta:=\alpha e^{\lambda_0t_0}+M_2e^{\lambda_1t_0}+M_3e^{\lambda_0t_0}<1$. Then, the fractal dimension of global attractor $\mathcal{A}$  has an upper bound
 \begin{equation}\label{3.34}
\operatorname{dim}_f \mathcal{A}\leq \frac{\ln\Lambda +\Lambda\ln(2+\frac{2 M_1}{\alpha})}{-\ln \zeta}<\infty,
\end{equation}
where $\Lambda$ is the dimension of $PX$ defined by \eqref{2.2} and $M_1, M_2, M_3, \lambda_0$ and $\lambda_1$ are given in $\mathbf{Hypothesis\  A1}$. 
\end{thm}
\begin{proof}
Since $\mathcal{A}$ is a global attractor, then it is compact and hence the number $R_\mathcal{A}$ is well defined. Thus, for any $u_0 \in \mathcal{A}$, we have
 \begin{equation}\label{3.35}
\mathcal{A} \subseteq B\left(u_0, R_\mathcal{A}\right),
\end{equation}
where $B\left(u_0,R_\mathcal{A}\right)$ is the ball with center $u_0$ and radius $R_\mathcal{A}$. For any $u \in \mathcal{A} \cap B\left(u_0, R_\mathcal{A}\right)$,  it follows from $\mathbf{Hypothesis\  A1}$ that
\begin{equation}\label{3.36}
\begin{gathered}
\left\|P S\left(t_0\right) u-P S\left(t_0\right) u_0\right\| \leq M_1e^{\lambda_0t_0}R_\mathcal{A},
\end{gathered}
\end{equation}
and
 \begin{equation}\label{3.37}
\begin{gathered}
\left\|(I-P) S\left(t_0\right) u-(I-P) S\left(t_0\right) u_0\right\| \leq  M_2e^{\lambda_1t_0}+M_3e^{\lambda_0t_0}R_\mathcal{A}.
\end{gathered}
\end{equation}
By Lemma \ref{lem3.1}, we can find $y_0^1, \ldots, y_0^{n_0}$ such that
 \begin{equation}\label{3.38}
\begin{gathered}
B_{P X}\left(P S\left(t_0\right) u_0, e^{\lambda_0t_0 }M_1R_\mathcal{A}\right) \subset \bigcup_{j=1}^{n_0} B_{PX}\left(y_0^j, \alpha e^{\lambda_0t_0}R_\mathcal{A}\right)
\end{gathered}
\end{equation}
with
 \begin{equation}\label{3.39}
\begin{gathered}
n_0\leq \Lambda 2^\Lambda \left(1+\frac{M_1}{\alpha}\right)^\Lambda,
\end{gathered}
\end{equation}
where $\Lambda$ is the dimension of $P X$.
\newline
Set
 \begin{equation}\label{3.40}
\begin{gathered}
u_0^j=y_0^j+(I-P) S\left(t_0\right) u_0
\end{gathered}
\end{equation}
for $ j=1, \ldots, n_0$. Then, for any $u \in \mathcal{A} \cap B\left(u_0, R_\mathcal{A}\right)$, there exists $j$ such that
 \begin{equation}\label{3.41}
\begin{aligned}
\left\|S\left(t_0\right) u-u_0^j\right\|
& \leq\left\|P S\left(t_0\right) u-y_0^j\right\|+\left\|(I-P) S\left(t_0\right) u-(I-P) S\left(t_0\right) u_0\right\| \\
& \leq\left(\alpha e^{\lambda_0 t_0}+M_2e^{\lambda_1 t_0}+M_3e^{\lambda_0 t_0}\right)R_\mathcal{A}.
\end{aligned}
\end{equation}
Since $\mathcal{A}$ is invariant,  i.e., $\mathcal{A}=S\left(t_0\right) \mathcal{A}$, we have
 \begin{equation}\label{3.43}
\begin{aligned}
\mathcal{A} & =S\left(t_0\right)\left(\mathcal{A}\cap B\left(u_0, R_\mathcal{A}\right)\right) & \subseteq \bigcup_{j=1}^{n_0} B\left(u_{0}^j,\left(\alpha e^{\lambda_0t_0}+M_2e^{\lambda_1t_0}+M_3e^{\lambda_0t_0}\right)R_\mathcal{A}\right).
\end{aligned}
\end{equation}
Denote by $\zeta=(\alpha e^{\lambda_0t_0}+M_2e^{\lambda_1t_0}+M_3e^{\lambda_0t_0})$. Applying the formula recursively for $k$ times gives
 \begin{equation}\label{3.44}
\begin{aligned}
\mathcal{A} & =S\left(k t_0 \right)\left(\mathcal{A}\cap B\left(u_0, R_\mathcal{A}\right)\right) & \subseteq \bigcup_{j=1}^{n_0,n_1,\cdots, n_{k-1}} B\left(u_{k-1}^j,\zeta^k R_\mathcal{A}\right),
\end{aligned}
\end{equation}
implying that the minimal number $N_{r_k}\left(\mathcal{A}\right)$ of balls with radius $r_k=\zeta^k R_\mathcal{A}$ covering $\mathcal{A}$ in $X$ satisfies
 \begin{equation}\label{3.44}
\begin{aligned}
N_{r_k}\left(\mathcal{A}\right) \leq n_0 \cdot \ldots \cdot n_{k-1} \leq[\Lambda 2^\Lambda \left(1+
\frac{M_1}{\alpha}\right)^\Lambda]^k.
\end{aligned}
\end{equation}
Since we have assumed that $\zeta <1$, then $r_k\rightarrow 0$ as $k\rightarrow \infty$.
Then it follows from \eqref{3.44} that
 \begin{equation}\label{3.45}
\begin{aligned}
\operatorname{dim}_f \mathcal{A} & =\limsup_{r_k\rightarrow 0} \frac{\ln N_{r_k}(\mathcal{A})}{-\ln r_k}\\
& \leq \limsup_{k\rightarrow \infty} \frac{\ln [\Lambda 2^\Lambda \left(1+\frac{M_1}{\alpha}\right)^\Lambda]^k}{-\ln (\zeta^k R_\mathcal{A})}\\
&=\frac{\ln\Lambda +\Lambda\ln(2+\frac{2 M_1}{\alpha})}{-\ln \zeta}<\infty.
\end{aligned}
\end{equation} 
\end{proof}
\begin{rem}\label{rem2.2} 
As noticed in Remark \ref{rem2.1}, the  estimation of $d_f(\mathcal{A})$ depends on the parameter $\alpha$. The same arguments described there are also valid here and, therefore, in some situations in which $M_1e^{\lambda_0t_0}$ is small enough we can deduce a bound for the fractal dimension which is independent of $\alpha$. Namely, if we assume   $M_1e^{\lambda_0t_0}+M_2e^{\lambda_1t_0}+M_3e^{\lambda_0t_0}<1$, then for all $\alpha\in (0,M_1)$, we have
$\alpha  e^{\lambda_0t_0}+M_2e^{\lambda_1t_0}+ M_3e^{\lambda_0t_0}<1$  and hence we deduce, taking $\alpha\uparrow M_1$, that
 \begin{equation}\label{3.4x}
\operatorname{dim}_f \mathcal{A}\leq \frac{\ln\Lambda +\Lambda\ln 4}{-\ln(M_1e^{\lambda_0t_0}+M_2e^{\lambda_1t_0}+M_3e^{\lambda_0t_0})}<\infty,
\end{equation}
which is independent of $\alpha$. 
\end{rem}
\section{Dimensions of pullback attractors for non-autonomous evolution process}
This section is devoted to the investigation of Hausdorff and fractal dimensions of pullback attractors. We first introduce some definitions and preliminaries about evolution processes, pullback attractors as well as their existence criteria.
\begin{defn}\label{defn4.1}
A family of two-parameter  mappings $\{S(t, s): t, s \in \mathbb{R}, t \geqslant s\}$ acting on $X$ is said to be an evolution process on $X$ if it satisfies
 \begin{equation}\label{4.1}
\begin{aligned}
S(t, \tau) S(\tau, s) & =S(t, s), \quad \forall  t, \tau, s \in \mathbb{R}, \quad t \geqslant \tau \geqslant s, \\
S(s, s) & =\mathrm{Id}_X, \quad \forall s \in \mathbb{R},
\end{aligned}
\end{equation}
where $\mathrm{Id}_X: X \rightarrow X$ represents the identity map on $X$.
\end{defn}

For notation simplicity, we will write $\{S(t, s): t, s \in \mathbb{R}, t \geqslant s\}$ simply as $\{S(t, s)\}$ in the following.
The notion of  a pullback attractor  is closely related to the following definition of a pullback absorbing set.
\begin{defn}\label{defn4.5}
The family $\{\mathcal{B}(t)\}_{t \in \mathbb{R}}$ is said to be  pullback  absorbing with respect to the process $\{S(t, s)\}$ if, for all $t \in \mathbb{R}$ and all $D \subset X$ bounded, there exists $T_D(t)>0$ such that for all $s \geqslant T_D(t)$
 \begin{equation}\label{4.7}
S(t, t-s) D \subset \mathcal{B}(t).
\end{equation}
The absorption is said to be uniform if $T_D(t)$ does not depend on the time variable $t$.
\end{defn}

\begin{defn}\label{defn4.4} Let $\{S(t, s)\}$ be a process on a Banach space $X$. A family of compact sets $\{\mathcal{A}(t)\}_{t \in \mathbb{R}}$ is said to be a  pullback attractor for $S$ if, for all $\tau \in \mathbb{R}$, it satisfies\\
(i) $S(t, \tau) \mathcal{A}(\tau)=\mathcal{A}(t)$ for all $t \geqslant \tau$;\\
(ii) $\lim _{s \rightarrow \infty} \operatorname{dist}(S(t, t-s) D, \mathcal{A}(t))=0$ for all bounded subsets $D$ of $X$.\\
The pullback attractor is said to be uniform if the attraction property is uniform in time, i.e.
$$\lim _{s \rightarrow \infty} \sup _{t \in \mathbb{R}} \operatorname{dist}(S(t, t-s) D, \mathcal{A}(t))=0$$
for all bounded subsets $D \subset X$.
\end{defn}
In the definition, $\operatorname{dist}(A, B)$ is the Hausdorff semidistance between $A$ and $B$, defined as
 \begin{equation}\label{4.6}
\operatorname{dist}(A, B)=\sup _{a \in A} \inf _{b \in B} d(a, b), \text { for } A, B \subseteq X,
\end{equation}
 where $d(a, b)=\|b-a\|_{X}$. 

Indeed, just as in the autonomous case, the existence of compact absorbing sets is the crucial property in order to obtain pullback attractors. For the following result, see Babin and Vishik \cite{2} or Temam \cite{37}.
\begin{lem}\label{lem4.1} Let $\{S(t, \tau)\}$ be a two-parameter process, and suppose $S(t, \tau): X \rightarrow X$ is continuous for all $t \geqslant \tau$. If there exists a family of compact pullback  absorbing sets $\{\mathcal{B}(t)\}_{t \in \mathbb{R}}$, then there exists a pullback attractor $\{\mathcal{A}(t)\}_{t \in \mathbb{R}}$, and $\mathcal{A}(t) \subset \mathcal{B}(t)$ for all $t \in \mathbb{R}$. Furthermore,
$$\mathcal{A}(t)=\overline{\bigcup_{D \subset X} \Lambda_D(t)},$$
where
$$
\Lambda_D(t)=\bigcap_{n \in \mathbb{N}} \overline{\bigcup_{s \geqslant n} S(t, t-s) D}
$$
and $D$ is bounded.
\end{lem}

We first study the Hausdorff dimension. Similar to $\mathbf{Hypothesis\  A1}$, we impose the following conditions on $\{S(t,s)\}$.

$\mathbf{Hypothesis\  A2}$ There is a finite dimensional projection $P(t):X\rightarrow PX$ with a finite dimension \begin{equation}\label{2.2a}
\Lambda=\dim\{P(t)X\}
\end{equation}
and there are three positive numbers $M_1, M_2, M_3$ and two constants $\lambda_0$ and $\lambda_1$ such that
\begin{equation}\label{3.2a}
\left\|P(t) S(t,t-s_0)\varphi-P(t)S(t,t-s_0)\psi\right\| \leq M_1e^{\lambda_0 s_0}\left\|\varphi-\psi\right\|
\end{equation}
and
 \begin{equation}\label{3.3a}
\begin{gathered}
\left\|(I-P(t)) S(t,t-s_0)\varphi-(I-P(t)) S(t,t-s_0)\psi\right\|\leq (M_2e^{\lambda_1 s_0}+M_3e^{\lambda_0 s_0})\left\|\varphi-\psi\right\|
\end{gathered}
\end{equation}
for any $t\in \mathbb{R}$ and some $s_0\geq 0$ and $\varphi, \psi$ in $\mathcal{A}(t)$. 

\begin{thm}\label{thm4.1} Assume that $\{S(t,s)\}$ is a continuous evolution process with a pullback attractor $\{\mathcal{A}(t)\}_{t\in \mathbb{R}}$  , $\mathbf{Hypothesis\  A2}$ holds and there exist  $0<\alpha<2$   such that
 \begin{equation}\label{3.4a}
\alpha M_1e^{\lambda_0 s_0}+2M_2e^{\lambda_1 s_0}+2M_3e^{\lambda_0 s_0}<1.
\end{equation}
Then, the Hausdorff dimension of the global attractor $\mathcal{A}(t)$ satisfies
 \begin{equation}\label{3.4b}
d_{H}<\frac{-\ln\Lambda-\Lambda\ln(2+\frac{4}{\alpha})}{\ln (\alpha M_1e^{\lambda_0 s_0}+2M_2e^{\lambda_1 s}+2M_3e^{\lambda_0 s_0})},
\end{equation} 
where $\Lambda$ is the dimension of $P(t)X$ defined  by \eqref{2.2a}  and $M_1, M_2, M_3, \lambda_0$ and $\lambda_1$ are given in $\mathbf{Hypothesis\  A1}$.
\end{thm}
\begin{proof}
 Replacing $\mathcal{A}$ and $S(t)$ in the proof of Theorem \ref{thm3.1} by  $\mathcal{A}(t-s_0)$ and $S(t,t-s_0)$ respectively till \eqref{3.8}. For any $t\in \mathbb{R}$ and the $s_0\in \mathbb{R}^+$ in $\mathbf{Hypothesis\  A2}$.  It follows from \eqref{3.2a} and \eqref{3.3a} that for any $u\in B\left(u_i, 2 r_i\right) \cap \mathcal{A}(t-s_0)$, we have
 \begin{equation}\label{3.9a}
\begin{gathered}
\left\|P(t) S(t,t-s_0) u-P(t)  S(t,t-s_0) u_i\right\| \leq 2 M_1e^{\lambda_0 s_0} r_i,
\end{gathered}
\end{equation}
and
 \begin{equation}\label{3.10a}
\begin{gathered}
\left\|(I-P(t) ) S(t,t-s_0) u-(I-P(t) ) S(t,t-s_0) u_i\right\| \leq 2(M_2e^{\lambda_1 s_0}+M_3e^{\lambda_0 s_0})r_i.
\end{gathered}
\end{equation}
By Lemma \ref{lem3.1}, for any $\alpha>0$, we can find $y_i^1, \ldots, y_i^{n_i}$ such that
 \begin{equation}\label{3.11a}
\begin{gathered}
B_{P(t) X}\left(P(t) S(t,t-s_0) u_i, 2 M_1e^{\lambda_0 s_0} r_i\right) \subset \bigcup_{j=1}^{n_i} B_{P(t) X}\left(y_i^j, \alpha M_1e^{\lambda_0 s_0} r_i\right)
\end{gathered}
\end{equation}
with
 \begin{equation}\label{3.12a}
\begin{gathered}
n_i \leq \Lambda 2^\Lambda \left(1+\frac{2}{ \alpha}\right)^\Lambda,
\end{gathered}
\end{equation}
where $\Lambda$ is the dimension of $P(t) X$.
\newline
Set
 \begin{equation}\label{3.13a}
\begin{gathered}
u_i^j=y_i^j+(I-P(t)) S(t,t-s_0) u_i
\end{gathered}
\end{equation}
for $i=1, \ldots, N, j=1, \ldots, n_i$. Then, for any $u\in B\left(u_i, 2 r_i\right) \cap \mathcal{A}(t_0-s)$, there exists $j$ such that
 \begin{equation}\label{3.14a}
\begin{aligned}
\left\|S(t,t-s_0) u-u_i^j\right\|
& \leq\left\|P(t) S(t,t-s_0) u-y_i^j\right\|+\left\|(I-P(t))S(t,t-s_0) u-(I-P(t))S(t,t-s_0) u_i\right\| \\
& \leq\left(\alpha M_1e^{\lambda_0 s_0}+2M_2e^{\lambda_1 s}+2M_3e^{\lambda_0 s_0}\right) r_i
\end{aligned}
\end{equation}
with
 \begin{equation}\label{3.16a}
n_i \leq \Lambda(2+\frac{4}{\alpha})^{\Lambda}.
\end{equation}
Denote by $\eta=(\alpha M_1e^{\lambda_0 s_0}+2M_2e^{\lambda_1 s_0}+2M_3e^{\lambda_0 s_0})$. Then we have
 \begin{equation}\label{3.20a}
S(t,t-s_0)\left(B\left(u_i, 2 r_i\right) \cap \mathcal{A}\right) \subset \bigcup_{j=1}^{n_i} B\left(u_i^j, \eta r_i\right).
\end{equation}
By the invariance property (i) in the definition \ref{defn4.4}, i.e., $\mathcal{A}(t)=S(t,t-s_0) \mathcal{A}(t-s_0)$, we have
 \begin{equation}\label{4.5}
\mathcal{A}(t) \subset \bigcup_{i=1}^N \bigcup_{j=1}^{n_i} B\left(u_i^j, \eta r_i\right) .
\end{equation}
This gives rise to, for any $d \geq 0$,
 \begin{equation}\label{4.6}
\begin{aligned}
\mu_H\left(\mathcal{A}(t), d, \eta\varepsilon\right)
\leq \sum_{i=1}^N \sum_{j=1}^{n_i} \eta^{d}r_i^d \leq \Lambda(2+\frac{4}{\alpha})^{\Lambda} \eta^{d} \sum_{i=1}^N r_i^d,
\end{aligned}
\end{equation}
we deduce by taking the infimum over all the coverings of $\mathcal{A}(t)$ by balls of radii less than $\varepsilon$:
 \begin{equation}\label{4.7}
\begin{aligned}
\mu_H\left(\mathcal{A}(t), d, \eta\varepsilon\right)\leq \Lambda(2+\frac{4}{\alpha})^{\Lambda} \eta^{d}\mu_H(\mathcal{A}(t-s_0), d, \varepsilon).
\end{aligned}
\end{equation}
Applying the formula recursively for $k$ times and it follows from the fact that $\mathcal{A}(t)=S(t,t-s_0)S(t,t-s_0)\cdots S(t,t-ks_0) \mathcal{A}(t-ks_0)$ we have
 \begin{equation}\label{4.8}
\begin{aligned}
\mu_H\left(\mathcal{A}(t), d, \eta\varepsilon\right)\leq [\Lambda(2+\frac{4}{\alpha})^{\Lambda} \eta^{d}]^k\mu_H(\mathcal{A}(t-ks_0), d, \varepsilon).
\end{aligned}
\end{equation}
Therefore, if
 \begin{equation}\label{4.9}
d_H<\frac{-\ln\Lambda-\Lambda\ln(2+\frac{4}{\alpha})}{\ln (\alpha M_1e^{\lambda_0 s_0}+2M_2e^{\lambda_1 s_0}+2M_3e^{\lambda_0 s_0})},
\end{equation}
then
 \begin{equation}\label{4.10}
\Lambda(2+\frac{4}{\alpha})^{\Lambda} \eta^{d}<1,
\end{equation}
and \eqref{4.8} leads to
 \begin{equation}\label{4.11}
\mu_H(\mathcal{A}(t), d, \varepsilon) \rightarrow 0,
\end{equation}
when $k \rightarrow \infty$. This completes the proof. 
\end{proof}
\begin{rem}\label{rem3.1} 
Taking into account the content of Remark \ref{rem2.1} we can obtain some estimation for the dimension which is independent on $\alpha$. Indeed, if we  take $\alpha\uparrow 2$ and assume   $2M_1e^{\lambda_0s_0}+2M_2e^{\lambda_1s_0}+2M_3e^{\lambda_0s_0}<1$, then for all $\alpha\in (0,2)$, we have
$\alpha M_1e^{\lambda_0s_0}+2M_2e^{\lambda_1s_0}+2M_3e^{\lambda_0s_0}<1$  and hence we deduce
 \begin{equation}\label{3.4w}
d_{H}\leq \frac{-\ln\Lambda-\Lambda\ln 4}{\ln (2M_1e^{\lambda_0 s_0}+2M_2e^{\lambda_1 s_0}+2M_3e^{\lambda_0 s_0})},
\end{equation}
which is independent of $\alpha$, although it may only be optimal for small values of $M_1e^{\lambda_0s_0}$. 
\end{rem}

Subsequently, we study the fractal dimension of pullback attractors for the evolution process $\{S(t,s)\}$.
\begin{thm}\label{thm4.2}
Let $\{\mathcal{A}(t)\}_{t\in \mathbb{R}}$ be the pullback attractor  of $\{S(t,s)\}$ with uniform finite diameter $R_{\mathcal{A}}:=\sup _{t\in \mathbb{R}}\sup _{u \in \mathcal{A}(t)}\|u\|_X <\infty$.  Assume that $\mathbf{Hypothesis\  A2}$ holds  and there  exists  $0<\alpha<M_1$  such that $\zeta:=\alpha e^{\lambda_0 s_0}+M_2e^{\lambda_1 s_0}+M_3e^{\lambda_0 s_0}<1$. Then, the fractal dimension of global attractor $\mathcal{A}(t)$  has an upper bound
 \begin{equation}\label{3.34}
\operatorname{dim}_f \mathcal{A}(t)\leq \frac{\ln\Lambda +\Lambda\ln(2+\frac{2 M_1}{\alpha})}{-\ln \zeta}<\infty,
\end{equation}
where $\Lambda$ is the dimension of $PX$ defined by \eqref{2.2} and $M_1, M_2, M_3, \lambda_0$ and $\lambda_1$ are given in $\mathbf{Hypothesis\  A2}$. 
\end{thm}
\begin{proof}
Since $\{\mathcal{A}(t)\}_{t\in \mathbb{R}}$ is a pullback attractor, then it is compact and hence the number $R_\mathcal{A}$ is well defined. Thus, for any $t\in \mathbb{R}$ and $u_0 \in \mathcal{A}(t)$, we have
 \begin{equation}\label{3.35a}
\mathcal{A}(t) \subseteq B\left(u_0, R_{\mathcal{A}}\right),
\end{equation}
where $B\left(u_0,R_\mathcal{A}\right)$ is the ball with center $u_0$ and radius $R_\mathcal{A}$. For any $u \in \mathcal{A}(t-s_0) \cap B\left(u_0, R_\mathcal{A}\right)$, it follows from $\mathbf{Hypothesis\  A2}$ that 
\begin{equation}\label{3.36a}
\begin{gathered}
\left\|P(t) S\left(t,t-s_0\right) u-P(t) S(t,t-s_0) u_0\right\| \leq M_1e^{\lambda_0s_0}R_\mathcal{A},
\end{gathered}
\end{equation}
and
 \begin{equation}\label{3.37a}
\begin{gathered}
\left\|(I-P(t)) S(t,t-s_0) u-(I-P(t)) S(t,t-s_0) u_0\right\| \leq  M_2e^{\lambda_1s_0}+M_3e^{\lambda_0s_0}R_\mathcal{A}.
\end{gathered}
\end{equation}
By Lemma \ref{lem3.1}, we can find $y_0^1, \ldots, y_0^{n_0}$ such that
 \begin{equation}\label{3.38a}
\begin{gathered}
B_{P(t) X}\left(P(t) S(t,t-s_0) u_0, e^{\lambda_0 s_0}M_1R_\mathcal{A}\right) \subset \bigcup_{j=1}^{n_0} B_{P(t)X}\left(y_0^j, \alpha e^{\lambda_0 s_0}R_\mathcal{A}\right)
\end{gathered}
\end{equation}
with
 \begin{equation}\label{3.39a}
\begin{gathered}
n_0\leq \Lambda 2^\Lambda \left(1+\frac{M_1}{\alpha}\right)^\Lambda,
\end{gathered}
\end{equation}
where $\Lambda$ is the dimension of $P(t)X$.
Set
 \begin{equation}\label{3.40a}
\begin{gathered}
u_0^j=y_0^j+(I-P(t)) S(t,t-s_0) u_0
\end{gathered}
\end{equation}
for $ j=1, \ldots, n_0$. Then, any $u \in \mathcal{A}(t-s_0) \cap B\left(u_0, R_\mathcal{A}\right)$, there exists $j$ such that
 \begin{equation}\label{3.41a}
\begin{aligned}
\left\|S(t,t-s_0) u-u_0^j\right\|
& \leq\left\|P(t)S(t,t-s_0) u-y_0^j\right\|+\left\|(I-P(t)) S(t,t-s_0) u-(I-P(t)) S(t,t-s_0) u_0\right\| \\
& \leq\left(\alpha e^{\lambda_0 s_0}+M_2e^{\lambda_1 s_0}+M_3e^{\lambda_0 s_0}\right)R_\mathcal{A}.
\end{aligned}
\end{equation}
Since $\mathcal{A}(t)$ is invariant,  i.e., $\mathcal{A}(t)=S(t,t-s_0) \mathcal{A}(t-s_0)$, we have
 \begin{equation}\label{3.43a}
\begin{aligned}
\mathcal{A}(t) & =S(t,t-s_0)\left(\mathcal{A}(t-s_0)\cap B\left(u_0, R_\mathcal{A}\right)\right) & \subseteq \bigcup_{j=1}^{n_0} B\left(u_{0}^j,\left(\alpha e^{\lambda_0 s_0}+M_2e^{\lambda_1 s_0}+M_3e^{\lambda_0 s_0}\right)R_\mathcal{A}\right).
\end{aligned}
\end{equation}
Denote by $\zeta=(\alpha e^{\lambda_0 s_0}+M_2e^{\lambda_1 s_0}+M_3e^{\lambda_0 s_0})$. Applying the formula recursively for $k$ times,
 \begin{equation}\label{3.44b}
\begin{aligned}
\mathcal{A} (t)& =S\left(t, t-s_0 \right)\cdots S\left(t-(k-1)s_0, t-k s_0 \right)\left(\mathcal{A}(t-ks_0)\cap B\left(u_0, R_\mathcal{A}\right)\right) & \subseteq \bigcup_{j=1}^{n_0,n_1,\cdots, n_{k-1}} B\left(u_{k-1}^j,\zeta^k R_\mathcal{A}\right),
\end{aligned}
\end{equation}
implying that the minimal number $N_{r_k}\left(\mathcal{A} (t)\right)$ of balls with radius $r_k=\zeta^k R_\mathcal{A}$ covering $\mathcal{A} (t)$ in $X$ satisfies
 \begin{equation}\label{3.44a}
\begin{aligned}
N_{r_k}\left(\mathcal{A} (t)\right) \leq n_1 \cdot \ldots \cdot n_k \leq\left[\Lambda 2^\Lambda \left(1+\frac{M_1}{\alpha}\right)^\Lambda\right]^k.
\end{aligned}
\end{equation}
Since we have assumed that $\zeta <1$, then $r_k\rightarrow 0$ as $k\rightarrow \infty$.
Then it follows from \eqref{3.44a} that
 \begin{equation}\label{3.45a}
\begin{aligned}
\operatorname{dim}_f\mathcal{A} (t) & =\limsup_{r_k\rightarrow 0} \frac{\ln N_{r_k}(\mathcal{A} (t))}{-\ln r_k}\\
& \leq \limsup_{k\rightarrow \infty} \frac{\ln [\Lambda 2^\Lambda \left(1+\frac{M_1}{\alpha}\right)^\Lambda]^k}{-\ln \zeta^k R_\mathcal{A}}\\
&=\frac{\ln\Lambda +\Lambda\ln(2+\frac{2M_1}{ \alpha})}{-\ln \zeta}<\infty.
\end{aligned}
\end{equation} 
\end{proof}
\begin{rem}\label{rem3.2} 
Again, at light of Remark \ref{rem2.2}, if we  take $\alpha\uparrow M_1$ and assume   $M_1e^{\lambda_0s_0}+M_2e^{\lambda_1s_0}+M_3e^{\lambda_0s_0}<1$, then for all $\alpha\in (0,M_1)$, we have
$\alpha  e^{\lambda_0s_0}+M_2e^{\lambda_1s_0}+ M_3e^{\lambda_0s_0}<1$  and hence we get an estimation
 \begin{equation}\label{3.4z}
\operatorname{dim}_f \mathcal{A}\leq \frac{\ln\Lambda +\Lambda\ln 4}{-\ln(M_1e^{\lambda_0s_0}+M_2e^{\lambda_1s_0}+M_3e^{\lambda_0s_0})}<\infty,
\end{equation}
which is independent of $\alpha$, although it may be only  optimal for small values of $M_1e^{\lambda_0s_0}$. 
\end{rem}
\section{Applications}
In this section, we are concerned about applications of the above established theoretical results to  the retarded reaction-diffusion equation and the non-autonomous retarded functional differential equations.
\subsection{Retarded reaction-diffusion equation}
This subsection is devoted to the Hausdorff and fractal dimensions of global attractors for an autonomous retarded reaction-diffusion equation on the bounded domain $[0, \pi]$ with a Dirichlet boundary condition.
 \begin{equation}\label{5.1}
 \left\{\begin{array}{l}
\frac{\partial}{\partial t} u(x, t)  =\frac{\partial^2}{\partial x^2} u(x, t)-au(x, t)-bu(x, t-r)+f(u(x, t-r)), 0 \leq x \leq \pi, t \geq 0, \\ u(0, t)  =u(\pi, t)=0, t \geq 0, \\ u(x, t) =\phi(t)(x), 0 \leq x \leq \pi,-r \leq t \leq 0,
\end{array}\right.
\end{equation}
where $a$, $b$ and $r$ are positive constants. Denote by $H=L^2(0, \pi)$ with inner product $(\xi,\eta)=\int_0^\pi \xi(x)\eta(x)dx$, norm $\|\xi\|_H=[\int_0^\pi \xi^2(x)dx]^{1/2}$ for any $\xi, \eta\in H$ and $X=C([-r, 0],H)$ the \textcolor{red}{space of continuous functions} from $[-r, 0]$ to $H$ endowed with the supremum norm $\|\phi\|=\sup_{\theta \in[-r, 0]}\|\phi(\theta)\|_H$ for any $\phi \in X$. In order to set the solution in the abstract semigroup framework, we define $A: H\rightarrow H$ by
 \begin{equation}\label{5.3}
\begin{aligned}
Ay=\ddot{y}
\end{aligned}
\end{equation}
with domain $\operatorname{Dom}\left(A\right)=\left\{y \in C^2([0, \pi]) ; y(0)=y(\pi)=0\right\}$, $L: X\rightarrow H$ by
 \begin{equation}\label{5.2}
\begin{aligned}
L\phi \triangleq -a\phi(0)-b\phi(-r)
\end{aligned}
\end{equation}
for any $\phi\in C$ and
$A_U: X\rightarrow X$ by
 \begin{equation}\label{5.4}
\begin{aligned}
& A_U\phi=A\phi(0)+L\phi\\
\end{aligned}
\end{equation}
for any $\phi\in X$.
It is well known that $A-aI$ generates an analytic compact semigroup $\{T(t)\}_{t \geq 0}$ on $H$ and  \cite{WJ} that $A_U$ generates a semigroup $\{U(t)\}_{t \geq 0}$. Moreover, we assume that $f$ satisfies the following global Lipschitz condition.

$\mathbf{Hypothesis\  A3}$
$\left\|f\left(\phi_1\right)-f\left(\phi_2\right)\right\|_{H} \leq L\left\|\phi_1-\phi_2\right\| \text { for any } \phi_1, \phi_2 \in X.$\\
It follows from \cite{WJ} Theorem 2.6 that \eqref{5.1} admits a global solution $u^\phi(\cdot):[-r, \infty] \rightarrow H$ such that $u^\phi(t)=\phi(t)$ for $t\in [-r,0]$ and
 \begin{equation}\label{5.5}
\begin{aligned}
u^\phi(t)=T(t) \phi(0)+\int_0^t T(t-s)\left[L\left(u_s^\phi\right)+f\left(u_s^\phi\right)\right] d s,
\end{aligned}
\end{equation}
for $t>0$.

Define $\Phi: \mathbb{R}\times X\rightarrow X$ by $\Phi(t)\phi=u^\phi_t(\cdot)$, then it generates an infinite dimensional dynamical system due to the uniqueness of the solution. The existence of attractor for semilinear or nonlinear partial functional differential equations including \eqref{5.1} as special case have been reported in much literature. See, for instance, \cite{YY} tackles the autonomous case with a nondensely defined linear part. Apparently, \eqref{5.1} satisfies other assumptions in \cite{YY} and hence it follows from Proposition 3.1 and Theorem 3.1 in \cite{YY} that \eqref{5.1} admits a global attractor which is stated as follows.
\begin{lem}\label{lem5.1}
Assume that $\mathbf{Hypothesis\  A3}$ holds. Then, for any $\phi \in X$, there exists a constant $\delta>a$ such that the integral solution $u^\phi_t(\cdot)$ of Eq. \eqref{5.1} satisfies the following inequality
 \begin{equation}\label{5.6}
\begin{aligned}
\left\|u_t\right\| \leq \frac{c_1 \mathrm{e}^{\delta r}}{a-L_f \mathrm{e}^{\delta r}}+\mathrm{e}^{\delta r}\left(\|\phi\|-\frac{c_1}{a-L_f \mathrm{e}^{\delta r}}\right) \mathrm{e}^{\left( L_f e^{\delta r}-a\right) t}, \quad t \geq 0,
\end{aligned}
\end{equation}
where $c_1=\|f(\mathbf{0})\|$, $a \neq L_f e ^{\delta r}$. If $a>L_f e^{\delta r}$, then Eq. \eqref{5.1} has a nonempty global attractor $\mathcal{A}$.
\end{lem}

We now estimate the dimensions of the global attractor  in Lemma \ref{lem5.1}. We first introduce the following \textcolor{red}{state decomposition results} of the linear part $A_U$ of \eqref{5.4} established in \cite{WJ}. It follows from  \cite{WJ} that the characteristic values of the linear part  $A_U$ are the roots of the following characteristic equation
 \begin{equation}\label{5.7}
\begin{aligned}
n^2 -\left(\lambda+a+b e^{-\lambda r}\right) =0, n=1,2, \cdots.
\end{aligned}
\end{equation}
Since $A_U$ is compact, it follows from Theorem 1.2 (i) in \cite{WJ} that the spectrum of $A_U$  are point spectra, which we denote by $\varrho_1>\varrho_2>\cdots$ with multiplicity $n_1, n_2,\cdots$, where $\varrho_1$ is defined as
 \begin{equation}\label{5.8}
\begin{aligned}
\varrho_1=\max \left\{\operatorname{Re} \lambda: n^2 -\left(\lambda+a+b e^{-\lambda r}\right) =0\right\}, n=1,2, \cdots.
\end{aligned}
\end{equation}
In the following, we always assume that $b-a<1$ and it follows from Lemma 1.13 on P73 in \cite{WJ} that if $a>0$, $b>0$ and $b-a<1$, then $\varrho_1<0$. For any given $\varrho_m<0$, $m\geq 1$, there is a
 \begin{equation}\label{5.8b}
\begin{aligned}
k_m=n_1+n_2+\cdots+n_m
\end{aligned}
\end{equation}
dimensional  subspace $X^U_{k_m}$ such that
$$X=X^U_{k_m} \bigoplus X^S_{k_m}$$
is the decomposition of $X$ by $\varrho_m$. Let $P_{k_m}$ and $Q_{k_m}$ be the projection of $X$ onto $X^U_{k_m}$ and $ X^S_{k_m}$ respectively, that is $X^U_{k_m}=P_{k_m}X$, $X^S_{k_m}=(I-P_{k_m})X=Q_{k_m}X$. It follows from the definition of $P_{k_m}$ and $Q_{k_m}$ that
\begin{equation}\label{5.8a}
\begin{aligned}
\left\|U(t)Q_{k_m} x\right\| & \leq K e^{\varrho_m t}\|x\|, & & t \geq0,
\end{aligned}
\end{equation}
where $K$ is a positive constant.

To show the squeezing property, we extend the domain of $U(t)$ to the following space of some discontinuous functions
 \begin{equation}\label{5.9}
\begin{aligned}
\hat{C}=\left\{\phi:[-r, 0] \rightarrow X ;\left.\phi\right\|_{[-r, 0)} \quad \text {is continuous and } \lim _{\theta \rightarrow 0^{-}} \phi(\theta) \in X \quad \text{exists} \right\}
\end{aligned}
\end{equation}
and introduce the following informal variation of  constant formula established in \cite{WJ}
 \begin{equation}\label{5.10}
\begin{aligned}
u(t) & =U(t) \phi+\int_0^t\left[U(t-s) X_0 f(u_s)\right](0) d s, \quad t \geq 0.
\end{aligned}
\end{equation}
It is proved by Theorem 2.1. in \cite{WJ} that the function $u:[-r, \infty) \rightarrow X$ defined by \eqref{5.5} satisfies \eqref{5.10} with $u_0=\phi\in X$, where $X_0:[-r, 0] \rightarrow B(X, X)$ is given by $X_0(\theta)=0$ if $-r \leq \theta<0$ and $X_0(0)=I d$.

\begin{rem}
In general, the solution semigroup defined by \eqref{5.10} have no definition at discontinuous functions and the integral in the formula is undefined as an integral in the phase space. However,  if interpreted correctly, \eqref{5.10} does make sense. Details can be found in \cite{CM} Pages 144 and 145.
\end{rem}

We can now prove the squeezing property from which it is clear that the global attractors of \eqref{5.1} have finite Hausdorff and fractal dimensions by Theorems \ref{thm3.1} and \ref{thm3.2}.
\begin{thm}\label{thm5.1}Let $P$ be  the finite dimension projection $P_{k_m}$ defined by \eqref{5.8a}, $\varrho_{1}, \varrho_{m}$  and $K$ being defined in \eqref{5.8} and \eqref{5.8a} respectively, then we have
 \begin{equation}\label{5.12}
\left\|P \Phi(t)\varphi-P \Phi(t)\psi\right\| \leq 2e^{(L_f+\varrho_1)t}\left\|\varphi-\psi\right\|
\end{equation}
and
 \begin{equation}\label{5.13}
\begin{gathered}
\left\|(I-P) \Phi(t)\varphi-(I-P) \Phi(t)\psi\right\|\leq (Ke^{\varrho_m t}+\frac{KL_f }{\varrho_1+L_f-\varrho_m} e^{(L_f+\varrho_1)t})\left\|\varphi-\psi\right\|
\end{gathered}
\end{equation}
for any $t\geq 0$ and $\varphi, \psi\in\mathcal{A}$.
\end{thm}
\begin{proof}
For any $\varphi, \psi\in X$, denote by $y=\varphi-\psi$ and $w_t=\Phi(t)\varphi-\Phi(t)\psi=u^\varphi_t-u^\psi_t$. Then it follows from \eqref{5.10} that
 \begin{equation}\label{5.10a}
\begin{aligned}
w_t& =U(t) y+\int_0^t U(t-s) X_0 [f(u^\varphi_s)-f(u^\psi_s)]d s, \quad t \geq 0.
\end{aligned}
\end{equation}
Taking projection $I-P$ on both sides of \eqref{5.10a} leads to
\begin{equation}\label{5.14}
\begin{aligned}
\|(I-P)w_t\|_X=&\|(I-P)U(t)y+\int_0^t (I-P)U(t-s) X_0 [f(u^\varphi_s)-f(u^\psi_s)]d s\|\\
\leq & K e^{\varrho_m t}\|y\| +L_f \int_0^t  e^{\varrho_1(t-s)}\|(I-P)w_s\|d s.
\end{aligned}
\end{equation}
Multiplying both sides of \eqref{5.14} by $e^{-\varrho_1 t}$,
\begin{equation}\label{5.15}
\begin{aligned}
e^{-\varrho_1 t}\|(I-P)w_t\| \leq & Ke^{(\varrho_m-\varrho_1) t}\|y\| +L_f \int_0^t  e^{-\varrho_1s}\|(I-P)w_t\|d s.
\end{aligned}
\end{equation}
By applying the Gronwall inequality, we have
\begin{equation}\label{5.16}
\begin{aligned}
e^{-\varrho_1 t}\|(I-P)w_t\| \leq & \|y\|[Ke^{(\varrho_m-\varrho_1) t} +\frac{KL_f }{\varrho_m-\varrho_1-L_f}(e^{(\varrho_m-\varrho_1) t}-e^{L_f})],
\end{aligned}
\end{equation}
indicating that
\begin{equation}\label{5.17}
\begin{aligned}
\|(I-P)w_t\| \leq & \|y\|[Ke^{\varrho_m t} +\frac{KL_f }{\varrho_m-\varrho_1-L_f}(e^{\varrho_m  t}-e^{(L_f+\varrho_1)t})]\\
 \leq & \|y\|[Ke^{\varrho_m t} +\frac{KL_f }{\varrho_1+L_f-\varrho_m}e^{(L_f+\varrho_1)t}].
\end{aligned}
\end{equation}
Hence, the second part holds with $\lambda_0=L_f+\varrho_1$, $\lambda_1=\varrho_m$, $M_2=K$ and $M_3=\frac{KL_f }{-\varrho_m+\varrho_1+L_f}$.

Subsequently, we prove the first part. Since $U(t)y=PU(t)y+(I-P)U(t)y$, we have
\begin{equation}\label{5.18}
\begin{aligned}
\|PU(t)y\|\leq &\|U(t)y\|+\|(I-P)U(t)y\|.
\end{aligned}
\end{equation}
\textcolor{red}{Taking projection of $P$ on both sides of \eqref{5.10a} gives
\begin{equation}\label{5.19}
\begin{aligned}
\|Pw_t\|\leq &\|PU(t)y\|+\int_0^t \|PU(t-s) X_0 [f(u^\varphi_s)-f(u^\psi_s)]d s\|\\
\leq & \frac{\|P\|}{\|I-P\|}\|(I-P)U(t)y\| \|U(t)y\|+L_f \int_0^t  e^{\varrho_1(t-s)}\|Pw_t\|d s\\
\leq &  \frac{|\varrho_{m}|}{|\varrho_{m+1}|}e^{\varrho_1t}\|y\|+L_f \int_0^t  e^{\varrho_1(t-s)}\|Pw_t\|d s.
\end{aligned}
\end{equation}
Multiplying both sides of \eqref{5.19} by $e^{-\varrho_1 t}$ implies
\begin{equation}\label{5.20}
\begin{aligned}
e^{-\varrho_1 t}\|Pw_t\| \leq &  \frac{|\varrho_{m}|}{|\varrho_{m+1}|}\|y\| +L_f \int_0^t  e^{-\varrho_1s}\|Pw_t\|d s.
\end{aligned}
\end{equation}
By applying the Gronwall inequality, we have
\begin{equation}\label{5.21}
\begin{aligned}
e^{-\varrho_1 t}\|Pw_t\| \leq &  \frac{|\varrho_{m}|}{|\varrho_{m+1}|}\|y\|e^{L_f t},
\end{aligned}
\end{equation}
which means that
\begin{equation}\label{5.22}
\begin{aligned}
\|Pw_t\| \leq &  \frac{|\varrho_{m}|}{|\varrho_{m+1}|}\|y\|e^{(L_f+\varrho_1) t}.
\end{aligned}
\end{equation}
Hence, the first part holds by taking $M_1= \frac{|\varrho_{m}|}{|\varrho_{m+1}|}$ and $\lambda_0=L_f+\varrho_1$.}
\end{proof}

It follows from Theorem \ref{thm3.1} that we have the following results about dimension of attractor $\{\mathcal{A}(t)\}_{t\in \mathbb{R}}$ for \eqref{5.1}.
\begin{thm}\label{thm5.2}\textcolor{red} {Let $k_m, \varrho_{1}, \varrho_{m}$ and $K$ be  defined in \eqref{5.8}, \eqref{5.8b} and \eqref{5.8a} respectively,  $P$ be  the finite dimensional projection $P_{k_m}$ defined by \eqref{5.8a}. Assume that conditions of Lemma \ref{lem5.1} are satisfied. Moreover, assume there exist  $0<\alpha<2$ and $t_0>0$ such that
 \begin{equation}\label{3.4g}
\alpha \frac{|\varrho_{m}|}{|\varrho_{m+1}|}e^{(L_f+\varrho_1) t_0}+2Ke^{\varrho_m t_0}+2\frac{KL_f }{\varrho_1+L_f-\varrho_m}e^{(L_f+\varrho_1) t_0}<1.
\end{equation}
Then, the Hausdorff dimension of the global attractor $\mathcal{A}$ satisfies
 \begin{equation}\label{3.4h}
d_{H}<\frac{-\ln k_m-k_m\ln(2+\frac{4}{\alpha})}{\ln (\alpha \frac{|\varrho_{m}|}{|\varrho_{m+1}|}e^{(L_f+\varrho_1) t_0}+2Ke^{\varrho_m t_0}+2\frac{KL_f }{\varrho_1+L_f-\varrho_m}e^{(L_f+\varrho_1) t_0})}.
\end{equation}}
\end{thm}
\begin{rem}\label{rem5.2}
Since $\varrho_1$ and $\varrho_m$ represent the first and the $m$-th eigenvalues of the linear part $A_U$ of Eq. \eqref{5.1}, which depends on the delay effect, we can see the Hausdorff dimension of global attractor $\mathcal{A}$ of Eq. \eqref{5.1} depends on the time delay via the distribution of eigenvalues of the linear part $A_U$ of Eq. \eqref{5.1}. Furthermore, it follows from \eqref{3.4h} that the Hausdorff dimension depends on the constants of exponential dichotomy, the Lipschitz constant of the nonlinear term and the spectrum gap of the linear part $A_U$, indicating that the Hausdorff dimension of global attractor $\mathcal{A}$ is very flexible to be tuned by a variety of parameters. \textcolor{red}{If we take $\alpha\uparrow 2$ and impose conditions as Remark \ref{rem2.1}, we can derive an estimation independent of  $\alpha$.}
\end{rem}

Particularly, we can see from \eqref{3.4h} that the Hausdorff dimension of global attractor $\mathcal{A}$ is monotone increasing with respect to  $m$. Thus, in the case $m=1$, we can obtain the minimum  Hausdorff dimension of global attractor $\mathcal{A}$, which is given in the following corollary.

\begin{cor}\label{cor5.1}
\textcolor{red}{ Let $\varrho_{1}$ and $K$ be  defined by  \eqref{5.8} and \eqref{5.8a} respectively and   $P$ be  the finite dimensional projection $P_{k_1}$ defined by \eqref{5.8a}. Assume that the conditions of Theorem \ref{thm5.1} are satisfied $a,b,r$ are appropriately chosen such that  $k_1=1$. Moreover, assume there exist  $0<\alpha<2$ and $t_0>0$ such that
 \begin{equation}\label{3.4ab}
\alpha \frac{|\varrho_{1}|}{|\varrho_{2}|}e^{(L_f+\varrho_1) t_0}+2Ke^{\varrho_1 t_0}+2Ke^{(L_f+\varrho_1) t_0}<1.
\end{equation}
Then, the Hausdorff dimension of global attractor $\mathcal{A}$ satisfies
 \begin{equation}\label{3.4ac}
d<\frac{-\ln (2+\frac{4}{\alpha})}{\ln (\alpha \frac{|\varrho_{1}|}{|\varrho_{2}|}e^{(L_f+\varrho_1) t_0}+2Ke^{\varrho_1 t_0}+2Ke^{(L_f+\varrho_1) t_0})}.
\end{equation}}
\end{cor}

By Theorem \ref{thm3.2}, we have the following results about the fractal dimension of Eq. \eqref{5.1}.

\begin{thm}\label{thm5.3}
\textcolor{red}{Let $k_m, \varrho_{1}, \varrho_{m},\gamma$ and $K$ be  defined in \eqref{5.8}, \eqref{5.8b} and \eqref{5.8a} respectively,  $P$ be  the finite dimensional projection $P_{k_m}$ defined by \eqref{5.8a}. Assume that the conditions of Lemma \ref{lem5.1} are satisfied. Moreover, assume there exist  $0<\alpha<\frac{|\varrho_{m}|}{|\varrho_{m+1}|}$ and $t_0>0$  such that $\zeta:=\alpha e^{(L_f+\varrho_1)t_0}+Ke^{\varrho_mt_0}+\frac{KL_f }{\varrho_1+L_f-\varrho_m}e^{(L_f+\varrho_1)t_0}<1$, then, the fractal dimension of global attractor $\mathcal{A}$  has an upper bound
 \begin{equation}\label{3.34f}
\operatorname{dim}_f \mathcal{A}\leq \frac{\ln k_m +k_m \ln(2+\frac{2|\varrho_{m}|}{\alpha |\varrho_{m+1}|})}{-\ln \zeta}<\infty.
\end{equation}}
\end{thm}

Similar to Remark \ref{rem5.2} and Corollary \ref{cor5.1}, we have the following corollary about the  fractal dimension of global attractor $\mathcal{A}$ in the case $\varrho_m=\varrho_1$.

\begin{cor}\label{cor5.1}
\textcolor{red}{ Let $\varrho_{1}$ and $K$ be  defined in \eqref{5.8}, \eqref{5.8b} and \eqref{5.8a} respectively and   $P$ be  the finite dimensional projection $P_{k_1}$  defined by \eqref{5.8a}. Assume that the conditions of Lemma \ref{lem5.1} are satisfied and $a,b,r$ are appropriately chosen such that  $k_1=1$. Moreover, assume  there exists $0<\alpha<t_0$ such that $\alpha e^{(L_f+\varrho_1)t_0}+Ke^{\varrho_1t_0}+K e^{(L_f+\varrho_1)t_0}<1$,
then, the fractal dimension of  global attractor $\mathcal{A}$ satisfies
 \begin{equation}\label{3.34g}
\operatorname{dim}_f \mathcal{A}\leq \frac{\ln(2+\frac{2 |\varrho_{1}|}{\alpha|\varrho_{2}|})}{-\ln [(\alpha+K) e^{(L_f+\varrho_1)t_0}+Ke^{\varrho_1t_0}]}<\infty.
\end{equation}}
\end{cor}

\subsection{Non-autonomous retarded functional differential  equation}
Consider the following typical non-autonomous RFDE arising in real world applications
 \begin{equation}\label{4.1}
\begin{aligned}
\dot{u}(t)=\sum_{k=1}^N A_k u\left(t-\omega_k\right)+\int_{-r}^0 A(t, \theta) u(t+\theta) d \theta+f(u_t), t\geq \sigma,\\
x_\sigma=\phi,
\end{aligned}
\end{equation}
where $0 \leq \omega_1<\omega_2<\cdots<\omega_N \leq r$, $u(t)\in \mathbb{R}^n$, $A_k\in \mathbb{R}^{n\times n}$,  $r$ stands for the delay.  The initial condition $\phi\in X\triangleq C([-r,0],\mathbb{R}^n)$, with $X$ being the Banach space of continuous functions from $[-r,0]$ to $\mathbb{R}^n$ equipped with the supremum norm  $\|\phi\|_{X}=\sup_{\theta \in[-r, 0]}|\phi(\theta)|$ for any $\phi \in X$ and $|\cdot|$ is the usual norm of $\mathbb{R}^n$. $A(t, \theta)$ is integrable in $\theta$ for each $t$ and there is a function $a \in \mathcal{L}_1^{\mathrm{loc}}((-\infty, \infty), \mathbb{R})$ such that
 \begin{equation}\label{3.1a}
\begin{aligned}
\left|\int_{-r}^0 A(t, \theta) \phi(\theta) d \theta\right| \leq a(t)|\phi|
\end{aligned}
\end{equation}
for all $t \geq \sigma$ and $\phi \in X$. $f$ is a  continuous nonlinear mapping from $X$ into $\mathbb{R}^n$

For notation simplicity, define the linear part of \eqref{4.1} as a  linear  mapping $L(t)$ from $X$ into $\mathbb{R}^n$ given by
 \begin{equation}\label{3.1b}
\begin{aligned}
L(t)\varphi =\sum_{k=1}^N A_k \varphi\left(-\omega_k\right)+\int_{-r}^0 A(t, \theta)\varphi(\theta) d \theta
\end{aligned}
\end{equation}
for any $\varphi\in X$. Following Chapter 6 in \cite{JH}, assume that there is an $n \times n$ matrix function $\eta(t, \theta)$, measurable in $(t, \theta) \in \mathbb{R} \times \mathbb{R}$, normalized so that
$$
\eta(t, \theta)=0 \quad \text { for } \quad \theta \geq 0, \quad \eta(t, \theta)=\eta(t,-r) \quad \text { for } \quad \theta \leq-r,
$$
$\eta(t, \theta)$ is continuous from the left in $\theta$ on $(-r, 0)$ and has bounded variation in $\theta$ on $[-r, 0]$ for each $t$. Further, there is an $m \in \mathcal{L}_1^{\text {loc }}((-\infty, \infty), \mathbb{R})$ such that
$$
\operatorname{Var}_{[-r, 0]} \eta(t, \cdot) \leq m(t)
$$
and the linear mapping $L(t): X \rightarrow \mathbb{R}^n$ is given by
$$
L(t) \varphi=\int_{-r}^0 d[\eta(t, \theta)] \varphi(\theta)
$$
for all $t \in(-\infty, \infty)$ and $\phi \in X$. Obviously, the norm of $L(t)$ satisfies $|L(t) \phi| \leq m(t)|\phi|$. It follows from Theorem 1.1 of Chapter 6 in \cite{JH} that under the above assumptions the following non-autonomous linear equation
 \begin{equation}\label{3.3}
\begin{aligned}
\dot{\tilde{u}}(t)=L(t) \tilde{u}_t,\\
\tilde{u}_\sigma=\phi
\end{aligned}
\end{equation}
admits a unique global solution $\tilde{u}^\phi(\cdot,\sigma): [\sigma-r, \infty)\rightarrow \mathbb{R}^n$ and hence the two parameters process $S(t,\sigma)$ on $X$ defined  by $S(t,\sigma)\phi=\tilde{u}_t^\phi(\cdot,\sigma)$ is a  continuous process. We always assume the trivial equilibrium $\tilde{u}=0$ of \eqref{3.3} is uniformly asymptotically stable and hence by Lemma 5.3 of Chapter 6 in \cite{JH},  there exist positive constants $\gamma$ and $K_{0}$ such that
\begin{equation}\label{4.35}\|S(t,\sigma)\phi\|<K_{0} \mathrm{e}^{-\gamma(t-\sigma)}\|\phi\|\end{equation}
 for all $t \geq \sigma$.

We  also assume the  nonlinear term $f$ satisfies $\mathbf{Hypothesis \  A3}$ and impose the following assumption similar to the exponential dichotomy on the process  $S(t,\sigma)$ generated by \eqref{3.3}.

$\mathbf{Hypothesis\  A4}$ There exist a positive constants $K$ and a negative constant $\beta<-\gamma$, and an $m$ dimensional projection operator $P(t): X \rightarrow X_m, s \in \mathbb{R}$ and $Q(s)=I-P(t):X \rightarrow X_m^\bot, s \in \mathbb{R}$ such that
\begin{equation}\label{4.36}
\begin{aligned}
\|Q(t)S(t, s)\|=\|S(t, s)Q(s)\| \leq K e^{\beta(t-s)}, \quad t \geq s
\end{aligned}
\end{equation}

\begin{rem}\label{rem4.1}
The $\mathbf{Hypothesis\  A4}$ is an intermediate and standard assumption in the study of dynamics of non-autonomous RFDEs, such as the boundary problem, the existence of almost periodic solutions \cite{JH} and invariant manifolds \cite{NW04}. In the autonomous case, it degenerates to a fact that the phase space $X$ can be decomposed into a finite dimensional unstable subspace and an infinite dimensional stable subspace, which is guaranteed by imposing some conditions on spectrum distribution of the linear operator. The details can be found in Chapter 7 of  \cite{JH} and similar results for retarded reaction diffusion equations in  Subsection 4.1. For non-autonomous operators, the spectrum condition under which  $\mathbf{Hypothesis\  A4}$ holds have been studied in the Appendix A of the very recent work \cite{BLM}.
\end{rem}

It follows from Theorem 1.2 of Chapter 6 in \cite{JH} and some standard contraction techniques, that under  assumption $\mathbf{Hypothesis \  A3}$, the non-autonomous nonlinear equation \eqref{4.1} admits a solution $u^\phi(t,\sigma)$ for any $t\in [\sigma-r, \infty)$, which is also continuous with respect to the initial condition. Define the non-autonomous evolution process  generated by \eqref{3.1} by $\Phi(t,\sigma)\phi=u^\phi_t(\cdot, \sigma)$ for any $\phi\in X$, which is continuous for any $t\geq \sigma$. In the following, we construct exponential attractors for of $\Phi(t,\sigma)$. We first show that $\Phi(t,\sigma)$ admits a family of positive invariant pullback absorbing sets $\mathcal{B}(\sigma)$ for any $\sigma\in \mathbb{R}$.

\begin{thm}\label{thm4.4}
Assume that $\mathbf{Hypothesis \  A4}$ as well as $\mathbf{Hypothesis \  A3}$ hold, $K_0<1$ and $K_0L_f-\gamma<0$. Then the dynamical system $\Phi$ admits an invariant pullback absorbing set $\mathcal{B}$ defined by
\begin{equation}\label{3.12}
\mathcal{B}=\{\phi \in C| \|\phi\|\leq \frac{1}{1-K_0}[\frac{K_0 f(0)}{\gamma}+\frac{1}{\gamma-K_0L_f}]\}.
\end{equation}
\end{thm}
\begin{proof}
By the following informal variation of  constant formula established in \cite{CM}
 \begin{equation}\label{3.10}
\begin{aligned}
u^\phi_t(\cdot, t-s) & =S(t,t-s) \phi+\int_{t-s}^t S(t, t-s-\rho) X_0 f(u^\phi_\rho(\cdot, t-\rho)) d \rho, \quad t \geq 0,
\end{aligned}
\end{equation}
where $X_0:[-r, 0] \rightarrow B(X, X)$ is given by $X_0(\theta)=0$ if $-r \leq \theta<0$ and $X_0(0)=I d$, we have
\begin{equation}\label{3.13}
\begin{aligned}
\left\|u^\phi_t(\cdot, t-s)\right\|\leq &\left\|\left(S(t,t-s)\phi\right)\right\|+\|\int_{t-s}^t S(t, t-(s-\rho)) X_0 f(u^\phi_\rho(\cdot, t-\rho)) d \rho\| \\
\leq & K_0e^{-\gamma s}\left\|\phi\right\|+K_0L_f\int_{t-s}^t e^{-\gamma (s-\rho) }(\|u^\phi_\rho(\cdot, t-\rho)\|+f(0)) \mathrm{d} \rho\\
\leq & K_0e^{-\gamma s}\left\|\phi\right\|+K_0L_f\int_{t-s}^t e^{-\gamma (s-\rho)}\|u^\phi_\rho(\cdot, t-\rho)\| \mathrm{d} \rho+\frac{K_0f(0)(1-e^{-\gamma s})}{\gamma}.
\end{aligned}
\end{equation}
Multiplying both sides of \eqref{3.13} by $e^{\gamma s}$,
\begin{equation}\label{3.14}
\begin{aligned}
 e^{\gamma s}\left\|u^\phi_t(\cdot, t-s)\right\|\leq & K_0 \left\|\phi\right\|+K_0L_f\int_{t-s}^t e^{\gamma \varrho}\|u^\phi_\rho(\cdot, t-\rho)\| \mathrm{d} \rho+\frac{K_0f(0)e^{\gamma s}}{\gamma}.
\end{aligned}
\end{equation}
Applying the Gr{o}nwall inequality yields
\begin{equation}\label{3.15}
\begin{aligned}
 e^{\gamma s}\left\|u^\phi_t(\cdot, t-s)\right\| \leq & K_0\left\|\phi\right\|e^{K_0L_f s}+\frac{K_0f(0)e^{\gamma s}}{\gamma}+\frac{e^{(\gamma-K_0L_f) s}}{\gamma-K_0L_f},
\end{aligned}
\end{equation}
and hence
\begin{equation}\label{3.16}
\begin{aligned}
\left\|u^\phi_t(\cdot, t-s)\right\| \leq & K_0\left\|\phi\right\|e^{(K_0L_f-\gamma) s}+\frac{K_0f(0)}{\gamma}+\frac{e^{-K_0L_f s}}{\gamma-K_0L_f}\\ \leq & K_0\left\|\phi\right\|e^{(K_0L_f-\gamma) s}+\frac{K_0f(0)}{\gamma}+\frac{1}{\gamma-K_0L_f}.
\end{aligned}
\end{equation}
Therefore, in the case $K_0L_f-\gamma<0$, for any $\phi\in X$, there exists a $s_{\|\phi\|}>0$ such that, for all $s\geq s_{\|\phi\|}$,
\begin{equation}\label{3.17}
\begin{aligned}
\left\|u^\phi_t(\cdot, t-s)\right\| \leq & \frac{1}{1-K_0}[\frac{K_0f(0)}{\gamma}+\frac{1}{\gamma-K_0L_f}].
\end{aligned}
\end{equation}
That is, $\mathcal{B}(t)$ is an absorbing set for $\Phi(t,t-s)$. Indeed, for any bounded subset $D\subset X$, denote by $r_D=\sup _{u \in D}\|u\|$, if we take $T_{D}=\frac{1}{\gamma}\ln \frac{r_D\gamma (1-K_0)(\gamma-K_0L_f)}{K_0f(0)(\gamma-K_0L_f)+\gamma}$, then we have
$$
\Phi(t,t-s) D \subset \mathcal{B}
$$
for all $s \geq T_{D}$.

The invariance property clearly follows since for any $\phi\in \mathcal{B}$, by \eqref{3.16} and \eqref{3.17}, we have
 \begin{equation}\label{3.18}
\begin{aligned}
\left\|\Phi(t,t-s)\phi\right\| =&\left\|u^\phi_t(\cdot, t-s)\right\|\leq K_0\left\|\phi\right\|e^{(K_0L_f-\gamma) s}+\frac{K_0f(0)}{\gamma}+\frac{e^{-K_0L_f s}}{\gamma-K_0L_f}
 \\ \leq & (\frac{K_0}{1-K_0}+1)[\frac{K_0f(0)}{\gamma}+\frac{1}{\gamma-K_0L_f}] \\ \leq & \frac{1}{1-K_0}[\frac{K_0f(0)}{\gamma}+\frac{1}{\gamma-K_0L_f}].
\end{aligned}
\end{equation}
This completes the proof.
\end{proof}

\begin{rem}
Theorem \ref{thm4.4} together  with continuity and compactness of the evolution process $\Phi(t,\sigma)$ implies it admits a pullback attractor $\{\mathcal{A}(t)\}_{t\in \mathbb{R}}$. In the sequel, we give upper bounds of Hausdorff and fractal dimensions
of  the evolution process $\Phi(t,\sigma)$ generated  by \eqref{4.1}.
\end{rem}

Subsequently, we prove the squeezing property of $\Phi$, i.e., $\left(\mathcal{H}_3\right)$ holds.
\begin{thm}\label{thm4.5}Let $P$ be  the finite dimensional projection $P_{k_m}$ defined by \eqref{4.35}, $K, \beta,\gamma$ and $K_0$ being defined by \eqref{4.35} and \eqref{4.36} respectively and assumptions of Theorem \ref{thm4.4} hold, then we have
 \begin{equation}\label{3.19}
\left\|P(t) \Phi(t,\sigma)\varphi-P(t) \Phi(t,\sigma)\psi\right\| \leq 2e^{(K_0L_f-\gamma) (t-\sigma)}\left\|\varphi-\psi\right\|
\end{equation}
and
 \begin{equation}\label{3.20}
\begin{gathered}
\left\|(I-P(t)) \Phi(t,\sigma)\varphi-(I-P(t)) \Phi(t,\sigma)\psi\right\|\leq (Ke^{\beta (t-\sigma)}+\frac{KL_fK_0}{-\gamma+L_f-\beta} e^{(L_f-\gamma)(t-\sigma)})\left\|\varphi-\psi\right\|
\end{gathered}
\end{equation}
for any $t\geq 0$ and $\varphi, \psi\in\mathcal{B}$.
\end{thm}
\begin{proof}
For any $\varphi, \psi\in X$, denote by $y=\varphi-\psi$ and $w_t(\cdot,\sigma)=\Phi(t,\sigma)\varphi-\Phi(t,\sigma)\psi=u^\varphi_t(\cdot, \sigma)-u^\psi_t(\cdot, \sigma)$. Then it follows from \eqref{3.10} that
 \begin{equation}\label{3.21}
\begin{aligned}
w_t(\cdot,\sigma)& =S(t,\sigma) y+\int_{t-s}^t S(t, s) X_0 [f(u^\varphi_s)-f(u^\psi_s)]d s, \quad t \geq 0.
\end{aligned}
\end{equation}
Taking projection $I-P(t)$ on both sides of \eqref{3.21} leads to
\begin{equation}\label{3.22}
\begin{aligned}
\|(I-P(t))w_t(\cdot,\sigma)\|=&\|(I-P(t))S(t,t-s)y+\int_{t-s}^t (I-P(t))S(t, s) X_0 [f(u^\varphi_s)-f(u^\psi_s)]d s\|\\
\leq & K e^{\beta (t-\sigma)}\|y\| +L_fK_0 \int_{t-s}^t  e^{-\gamma(t-s)}\|(I-P(t))w_s\|d s.
\end{aligned}
\end{equation}
Multiplying both sides of \eqref{3.22} by $e^{\gamma (t-\sigma)}$,
\begin{equation}\label{3,23}
\begin{aligned}
e^{\gamma (t-\sigma)}\|(I-P(t))w_t(\cdot,\sigma)\| \leq & Ke^{(\beta+\gamma) (t-\sigma)}\|y\| +L_fK_0 \int_{t-s}^t  e^{\gamma (s-\sigma)}\|(I-P(t))w_t(\cdot,\sigma)\|d s.
\end{aligned}
\end{equation}
By applying the Gronwall inequality, we have
\begin{equation}\label{3.24}
\begin{aligned}
e^{\gamma (t-\sigma)}\|(I-P(t))w_t(\cdot,\sigma)\| \leq & \|y\|[Ke^{(\beta+\gamma) (t-\sigma)} +\frac{KL_fK_0 }{\beta+\gamma-L_fK_0}(e^{(\beta+\gamma) (t-\sigma)}-e^{L_fK_0})],
\end{aligned}
\end{equation}
indicating that
\begin{equation}\label{3.25}
\begin{aligned}
\|(I-P(t))w_t(\cdot,\sigma)\| \leq & \|y\|[Ke^{\beta (t-\sigma)} +\frac{KL_fK_0}{\beta+\gamma-L_fK_0}(e^{\beta (t-\sigma)}-e^{(K_0L_f-\gamma) (t-\sigma)})]\\
 \leq & \|y\|[Ke^{\beta (t-\sigma)} +\frac{KL_fK_0 }{-\gamma+L_fK_0-\beta}e^{(K_0L_f-\gamma) (t-\sigma)}].
\end{aligned}
\end{equation}
Hence, the second part holds with $\lambda_0=L_fK_0-\gamma$, $\lambda_1=\beta$, $M_2=K$ and $M_3=\frac{KL_fK_0 }{-\beta-\gamma+L_fK_0}$.

Subsequently, we prove the first part. Since $S(t,\sigma)y=P(t)S(t,\sigma)y+(I-P(t))S(t,\sigma)y$, we have
\begin{equation}\label{3.26}
\begin{aligned}
\|P(t)S(t,\sigma)y\|\leq &\|S(t,\sigma)y\|+\|(I-P(t))S(t,\sigma)y\|.
\end{aligned}
\end{equation}
Taking projection of $P(t)$ on both sides of \eqref{3.10} and on account of \eqref{3.26},
\begin{equation}\label{3.27}
\begin{aligned}
\|P(t)w_t(\cdot,\sigma)\|=&\|S(t,\sigma)y\|+\|(I-P(t))S(t,\sigma)y\|+\int_{t-s}^t \|P(t)S(t, s) X_0 [f(u^\varphi_s)-f(u^\psi_s)]d s\|\\
\leq &(K_0 e^{-\gamma s}+Ke^{\beta (t-\sigma)})\|y\|+L_fK_0 \int_{t-s}^t  e^{-\gamma(t-s)}\|Pw_t(\cdot,\sigma)\|d s\\
\leq &(K_0+K)e^{-\gamma s}\|y\|+L_fK_0 \int_{t-s}^t  e^{-\gamma (t-s)}\|Pw_t(\cdot,\sigma)\|d s.
\end{aligned}
\end{equation}
Multiplying both sides of \eqref{3.28} by $e^{\gamma (t-\sigma)}$,
\begin{equation}\label{3.28}
\begin{aligned}
e^{\gamma (t-\sigma)}\|Pw_t(\cdot,\sigma)\| \leq &(K_0+K)\|y\| +L_fK_0\int_{t-s}^t  e^{\gamma (s-\sigma)}\|Pw_t(\cdot,\sigma)\|d s.
\end{aligned}
\end{equation}
By applying the Gronwall inequality, we have
\begin{equation}\label{3.29}
\begin{aligned}
e^{\gamma (t-\sigma)}\|Pw_t(\cdot,\sigma)\| \leq & (K_0+K)\|y\|e^{L_fK_0 (t-\sigma)},
\end{aligned}
\end{equation}
indicating that
\begin{equation}\label{3.30}
\begin{aligned}
\|Pw_t(\cdot,\sigma)\| \leq & (K_0+K)\|y\|e^{(L_fK_0-\gamma) (t-\sigma)}.
\end{aligned}
\end{equation}
Hence, the first part holds by taking $M_1=(K_0+K)$ and $\lambda_0=L_fK_0-\gamma$.
\end{proof}

 Apparently, $\lambda_1=\beta<-\gamma<L_fK_0-\gamma=\lambda_0$ in Theorem \ref{thm3.1}. Hence, it follows from Theorem \ref{thm3.1} that we have the following results about the dimension of pullback attractor $\{\mathcal{A}(t)\}_{t\in \mathbb{R}}$ of \eqref{5.1}.
\begin{thm}\label{thm4.6} \textcolor{red}{Let $K_0, \gamma, \beta$, $K$ and $P(t)$ be  defined in \eqref{4.35} and \eqref{4.36} respectively. Assume that the conditions of Theorems \ref{thm4.4} and \ref{thm4.5} are satisfied.
Moreover, assume there exist  $0<\alpha<2$ and $t_0>0$ such that
 \begin{equation}\label{3.4gb}
\alpha (K_0+K) e^{(L_fK_0-\gamma) t_0}+2Ke^{\beta t_0}+2\frac{L_fK_0-\gamma}{(L_fK_0-\gamma+\beta)}e^{(L_fK_0-\gamma) t_0}<1.
\end{equation}
Then, the Hausdorff dimension of the global attractor $\mathcal{A}$ satisfies
 \begin{equation}\label{3.4hc}
d_{H}<\frac{-\ln k_m-k_m\ln(2+\frac{4}{\alpha})}{\ln (\alpha (K_0+K) e^{(L_fK_0-\gamma) t_0}+2Ke^{\beta t_0}+2\frac{L_fK_0-\gamma}{(L_fK_0-\gamma+\beta)}e^{(L_fK_0-\gamma) t_0})}.
\end{equation}}
\end{thm}

Particularly, we can see from \eqref{3.4hc} the Hausdorff dimension of pullback attractor $\{\mathcal{A}(t)\}_{t\in \mathbb{R}}$ is monotone increasing with respect to  $m$. Thus, in the case $m=1$, we can obtain the minimum  Hausdorff dimension of pullback attractor $\{\mathcal{A}(t)\}_{t\in \mathbb{R}}$, which is given in the following corollary.

\begin{cor}\label{cor5.1}
  Let $K_0, \gamma, \beta$, $K$ and $P(t)$ be  defined in \eqref{4.35} and \eqref{4.36} respectively. Assume that the conditions of Theorems \ref{thm4.4} and \ref{thm4.5} are satisfied. \textcolor{red}{Moreover, assume there exist  $0<\alpha<2$ and $t_0>0$ such that
 \begin{equation}\label{3.4gb}
\alpha (K_0+K) e^{(L_fK_0-\gamma) t_0}+2Ke^{\beta t_0}+2\frac{L_fK_0-\gamma}{(L_fK_0-\gamma+\beta)}e^{(L_fK_0-\gamma) t_0}<1.
\end{equation}
Then, the Hausdorff dimension of the global attractor $\mathcal{A}$ satisfies
 \begin{equation}\label{3.4ahc}
d_{H}<\frac{-\ln(2+\frac{4}{\alpha})}{\ln (\alpha (K_0+K) e^{(L_fK_0-\gamma) t_0}+2Ke^{\beta t_0}+2\frac{L_fK_0-\gamma}{(L_fK_0-\gamma+\beta)}e^{(L_fK_0-\gamma) t_0})}.
\end{equation}}
\end{cor}

By Theorem \ref{thm3.2}, we have the following results about the fractal dimension of Eq. \eqref{5.1}.

\begin{thm}\label{thm5.3}
 Let $K_0, \gamma, \beta$, $K$ and $P(t)$ be  defined in \eqref{4.35} and \eqref{4.36} respectively. Assume that the conditions of Theorems \ref{thm4.4} and \ref{thm4.5} are satisfied. \textcolor{red}{Moreover, assume   there exist  $0<\alpha<K_0+K$ such that $\zeta:=\alpha e^{(L_fK_0-\gamma)t_0}+Ke^{\beta t_0}+\frac{KL_f }{L_fK_0-\gamma+\beta}e^{(L_fK_0-\gamma)t_0}<1$. Then, the fractal dimension of pullback attractor $\mathcal{A}(t)$  has an upper bound
 \begin{equation}\label{3.34f}
\operatorname{dim}_f \mathcal{A}(t) \leq \frac{\ln m +m\ln(2+\frac{2K_0+K}{\alpha})}{-\ln \zeta}<\infty.
\end{equation}}
\end{thm}

Similar to Remark \ref{rem5.2} and Corollary \ref{cor5.1}, we have the following corollary about the  fractal dimension of pullback attractor $\mathcal{A}(t)$ in the case $m=1$.

\begin{cor}\label{cor5.1}
 Let $K_0, \gamma, \beta$, $K$ and $P(t)$ be  defined in \eqref{4.35} and \eqref{4.36} respectively. Assume that the conditions of Theorems \ref{thm4.4} and \ref{thm4.5} are satisfied. \textcolor{red}{Moreover, assume   there exist  $0<\alpha<K_0+K$ such that $\zeta:=\alpha e^{(L_fK_0-\gamma)t_0}+Ke^{\beta t_0}+\frac{KL_f }{L_fK_0-\gamma+\beta}e^{(L_fK_0-\gamma)t_0}<1$.
Then, the fractal dimension of  pullback attractor $\mathcal{A}(t)$ satisfies
 \begin{equation}\label{3.34g}
\operatorname{dim}_f \mathcal{A}(t) \leq \frac{\ln(2+\frac{2K_0+K}{\alpha})}{-\ln \zeta}<\infty.
\end{equation}}
\end{cor}

\begin{rem}\label{rem4.3}
In \cite{JM76}, the author proved that the negative invariant sets of an autonomous retarded functional differential equation have finite fractal dimension by requiring differentiability of the solution semiflow. Furthermore, they recast the equation into a Hilbert space and  did not provide explicit bounds of the dimension. Here, we directly investigate the problem in the natural phase space, i.e., a Banach space and do not require the smooth condition. Moreover, we consider the non-autonomous case and establish an explicit estimation that only depends on the inner characteristic of the studied equation. \textcolor{red}{Apparently, if we take $\alpha\uparrow K_0+K$ and impose conditions as Remark \ref{rem3.2}, we can deduce an estimation independent of  $\alpha$.}
\end{rem}

\section{Conclusions}
In this paper, we have established some \textcolor{red}{new criterions} for the upper bounds of Hausdorff and fractal dimensions of global attractors and pullback attractors for both autonomous and  nonautonomous dynamical systems in Banach spaces. The methods show a wide applicability to infinite dimensional dynamical systems generated by functional differential equations, especially the ones of which the linear parts admit   exponential dichotomies with  decomposition of state spaces.  They may be also available to investigate topological dimensions of attractors for neutral partial functional differential equations, the infinite delay case as well as some other evolution equations with certain squeeze properties in Banach spaces.

In the applications, we only consider partial functional differential equations on bounded domain. Actually, there are many real world process evolution on infinite domain, such as the mature population of species living in an infinite habitat. In such a scenario, the Laplace operator has a continuous spectrum, $H^1(\mathbb{R}^n)$ is not compactly embedded in $L^2(\mathbb{R}^n)$ and the solution semiflow do not have absorbing sets that are compact in the original topology, causing the method  developed here no longer effective and new techniques should be established. This will be studied in an upcoming paper. Generally, random effects are omnipresent in mathematical modelings. Therefore, anther question is, what can we say about the topological dimensions of random attractors for partial functional differential equations perturbed by random effects, i.e. the stochastic functional differential equations(SFDEs). Indeed, even under what conditions do SFDEs generate random dynamical systems have not been perfectly tackled. This problem also deserves much efforts in the future.

\noindent{\bf Acknowledgement.}
The authors would like to  thank the anonymous referee  for his/her valuable comments and suggestions. This work was jointly supported by China Postdoctoral Science Foundation (2019TQ0089), China Scholarship Council(202008430247). \\
The research of T. Caraballo has been partially supported by Spanish Ministerio de Ciencia e
Innovaci\'{o}n (MCI), Agencia Estatal de Investigaci\'{o}n (AEI), Fondo Europeo de
Desarrollo Regional (FEDER) under the project PID2021-122991NB-C21.\\
This work was completed when Wenjie Hu was visiting the Universidad de Sevilla as a visiting scholar, and he would like to thank the staff in the Facultad de Matem\'{a}ticas  for their hospitality and thank the university for its excellent facilities and support during his stay.

\end{document}